\documentclass[12pt,letterpaper,final,twoside,leqno]{amsart}
\usepackage{amsmath,amssymb,amsthm,amsfonts,amscd,amsopn}
\usepackage{eucal,mathrsfs}
\usepackage{xspace,chngcntr}
\usepackage{rotating,mathtools}
\usepackage[all,cmtip,color]{xy}\xyoption{dvips}
 
\usepackage{comment} 
\usepackage{supertabular}
\usepackage
{hyperref}
\usepackage[dvipsnames,usenames]{xcolor}
\hypersetup{
    colorlinks=true,
    linkcolor={blue!50!black},
    citecolor={green!50!black},
    urlcolor={red!80!black}
}

\usepackage{datetime,chngcntr}
\usepackage{paralist}

\usepackage{enumitem}
\usepackage{mfirstuc}
\RequirePackage{hyperref}

\newcommand{\sectionsize}{\footnotesize}
\newcommand{\theoremsize}{\bfseries}

\renewcommand{\subsectionautorefname}{\sectionsize\sf \subsectionautorefname}

\makeatletter
\@ifdefinable\equationname{\let\equationname\equationautorefname}
\def\equationautorefname~#1\@empty\@empty\null{\protect{\theoremsize\sf
    (#1\@empty\@empty\null)}}%
\@ifdefinable\AMSname{\let\AMSname\AMSautorefname}
\def\AMSautorefname~#1\@empty\@empty\null{\rm (#1\@empty\@empty\null)}%
\@ifdefinable\itemname{\let\itemname\itemautorefname}
\def\itemautorefname~#1\@empty\@empty\null{\theoremsize\sf #1\@empty\@empty\null%
}%
\makeatother
\RequirePackage{amsthm}
\RequirePackage{aliascnt}

\newcommand{\basetheorem}[3]{%
    \newtheorem{#1}{#2}[#3]
    \newtheorem*{#1*}{#2}
    \expandafter\def\csname #1autorefname\endcsname{#2}
}%
\newcommand{\maketheorem}[3]{%
    \newaliascnt{#1}{#2}
    \newtheorem{#1}[#1]{\theoremsize\sf #3}
    \aliascntresetthe{#1}
    \expandafter\def\csname #1autorefname\endcsname{\theoremsize\sf #3}
    \newtheorem{#1*}{#3}
}%
\newcommand{\baseremark}[3]{%
    \newtheorem{#1}{#2}{#3}
    \newtheorem*{#1*}{#2}
    \expandafter\def\csname #1autorefname\endcsname{#2}
}%
\newcommand{\makeremark}[3]{%
    \newaliascnt{#1}{#2}
    \newtheorem{#1}[equation]{#3}
    \aliascntresetthe{#1}
    \expandafter\def\csname #1autorefname\endcsname{\theoremsize\sf #3}
    \newtheorem{#1*}{#3}
}%
\theoremstyle{plain}   
\basetheorem{theorem}{Theorem}{}
\usepackage{amsbsy}

\usepackage{marvosym}

\DeclareMathAlphabet{\smallchanc}{OT1}{pzc}%
                                 {m}{it}
\DeclareFontFamily{OT1}{pzc}{}
\DeclareFontShape{OT1}{pzc}{m}{it}%
             {<-> s * [1.100] pzcmi7t}{}
\DeclareMathAlphabet{\mathchanc}{OT1}{pzc}%
                                 {m}{it}

\newcommand{\sE}{\mathscr{E}}

\newcommand{\sL}{\mathscr{L}}

\newcommand{\sO}{\mathscr{O}}

\newcommand{\sX}{\mathscr{X}}

\newcommand{\bN}{\mathbb{N}}

\newcommand{\bP}{\mathbb{P}}
\newcommand{\bQ}{\mathbb{Q}}

\newcommand{\bZ}{\mathbb{Z}}

\DeclareSymbolFont{largesymbolsA}{U}{jkpexa}{m}{n}
\SetSymbolFont{largesymbolsA}{bold}{U}{jkpexa}{bx}{n}
\DeclareMathSymbol{\varprod}{\mathop}{largesymbolsA}{16}
\makeatletter
\newcommand{\LeftEqNo}{\let\veqno\@@leqno}
\makeatother
\newcommand{\properideal}%
        {\subsetneq}

\newcommand{\leteq}{\colon\!\!\!=}

\DeclareMathOperator{\kar}{char}

\DeclareMathOperator{\id}{{id}}

\DeclareMathOperator{\Pic}{Pic}

\newcommand{\factor}[2]{\left. \raise .2em\hbox{\ensuremath{#1}\vphantom{$I^d$}}
\hskip -.1em \right/ \hskip -.4em \raise -.3em\hbox{\ensuremath{#2}}}%
\newcommand\mtimes[3]{{\varprod_{#1}^{#2}}_{\raise 1ex \hbox{\scriptsize #3}}}%

\newcommand{\union}\cup
\newcommand{\intersect}\cap
\newcommand{\Union}\bigcup
\newcommand{\Intersect}\bigcap
\def\myoplus#1.#2.{\underset #1 \to {\overset #2 \to \oplus}}

\newcommand{\resto}[1]{\raise -.5ex\hbox{$\vert$}_{#1}}

\newcommand{\ses}{short exact sequence\xspace}
\newcommand{\sess}{short exact sequences\xspace}

\newcommand{\sings}{singularities\xspace}

\newcommand{\Fsp}{$F$-split\xspace}

\makeatletter
\definecolor{brick}{RGB}{204,0,0}
\def\@cite#1#2{{%
 \m@th\upshape\mdseries[{\small\sffamily #1}{\if@tempswa, \small\sffamily
   \color{brick} #2\fi}]}}

\setlist[enumerate]{itemsep=3pt,topsep=3pt,leftmargin=2em,label={\rm (\roman*)}}
\newlist{enumalpha}{enumerate}{1}
\setlist[enumalpha]{itemsep=3pt,topsep=3pt,leftmargin=2em,label=(\alph*\/)}
\newcounter{parentthmnumber}
\newcounter{currentparentthmnumber}
\newcounter{nexttag}

\newlist{thmlistaa}{enumerate}{1}
\setlist[thmlistaa]{label=(\arabic*), ref=\autoref{parentthm}\thethm(\arabic*)}
\newcommand*{\parentthmlabeldef}{%
  \expandafter\newcommand
  \csname parentthm\the\value{parentthmnumber}\endcsname
}
\newcommand*{\ptlget}[1]{%
  \romannumeral-`\x
  \ltx@ifundefined{parentthm\number#1}{%
    \ltx@space
    \parentthmundefined
  }{%
    \expandafter\ltx@space
    \csname mymacro\number#1\endcsname
  }%
}  
\newcommand*{\parentthmundefined}{\textbf{??}}
\parentthmlabeldef{parentthm}

\newlist{thmlistrr}{enumerate}{1}
\setlist[thmlistrr]{label=(\roman*), ref=\autoref{parentthm}(\roman*)}
\newcounter{proofstep}%
\newcounter{lastyear}
\addtocounter{lastyear}{-1}
\newcommand\noin{\noindent}

\newcommand\CM{Cohen-Macaulay\xspace}

\newtheoremstyle{bozont}{3pt}{3pt}%
     {\itshape}
     {}
     {\bfseries}
     {.}
     {.5em}
     {\thmname{#1}\thmnumber{ #2}\thmnote{ \rm #3}}
\newtheoremstyle{bozont-sub}{3pt}{3pt}%
     {\itshape}
     {}
     {\bfseries}
     {.}
     {.5em}
     {\thmname{#1}\ \arabic{section}.\arabic{thm}.\thmnumber{#2}\thmnote{ \rm #3}}
\newtheoremstyle{bozont-named-thm}{3pt}{3pt}%
     {\itshape}
     {}
     {\bfseries}
     {.}
     {.5em}
     {\thmname{#1}\thmnumber{#2}\thmnote{ #3}}
\newtheoremstyle{bozont-named-bf}{3pt}{3pt}%
     {}
     {}
     {\bfseries}
     {.}
     {.5em}
     {\thmname{#1}\thmnumber{#2}\thmnote{ #3}}
\newtheoremstyle{bozont-named-sf}{3pt}{3pt}%
     {}
     {}
     {\sffamily}
     {.}
     {.5em}
     {\thmname{#1}\thmnumber{#2}\thmnote{ #3}}
\newtheoremstyle{bozont-named-sc}{3pt}{3pt}%
     {}
     {}
     {\scshape}
     {.}
     {.5em}
     {\thmname{#1}\thmnumber{#2}\thmnote{ #3}}
\newtheoremstyle{bozont-named-it}{3pt}{3pt}%
     {}
     {}
     {\itshape}
     {.}
     {.5em}
     {\thmname{#1}\thmnumber{#2}\thmnote{ #3}}
\newtheoremstyle{bozont-sf}{3pt}{3pt}%
     {}
     {}
     {\sffamily}
     {.}
     {.5em}
     {\thmname{#1}\thmnumber{ #2}\thmnote{ \rm #3}}
\newtheoremstyle{bozont-sc}{3pt}{3pt}%
     {}
     {}
     {\scshape}
     {.}
     {.5em}
     {\thmname{#1}\thmnumber{ #2}\thmnote{ \rm #3}}
\newtheoremstyle{bozont-remark}{3pt}{3pt}%
     {}
     {}
     {\scshape}
     {.}
     {.5em}
     {\thmname{#1}\thmnumber{ #2}\thmnote{ \rm #3}}
\newtheoremstyle{bozont-subremark}{3pt}{3pt}%
     {}
     {}
     {\scshape}
     {.}
     {.5em}
     {\thmname{#1}\ \arabic{section}.\arabic{thm}.\thmnumber{#2}\thmnote{ \rm #3}}
\newtheoremstyle{bozont-def}{3pt}{3pt}%
     {}
     {}
     {\bfseries}
     {.}
     {.5em}
     {\thmname{#1}\thmnumber{ #2}\thmnote{ \rm #3}}
\newtheoremstyle{bozont-reverse}{3pt}{3pt}%
     {\itshape}
     {}
     {\bfseries}
     {.}
     {.5em}
     {\thmnumber{#2.}\thmname{ #1}\thmnote{ \rm #3}}
\newtheoremstyle{bozont-reverse-sc}{3pt}{3pt}%
     {\itshape}
     {}
     {\scshape}
     {.}
     {.5em}
     {\thmnumber{#2.}\thmname{ #1}\thmnote{ \rm #3}}
\newtheoremstyle{bozont-reverse-sf}{3pt}{3pt}%
     {\itshape}
     {}
     {\sffamily}
     {.}
     {.5em}
     {\thmnumber{#2.}\thmname{ #1}\thmnote{ \rm #3}}
\newtheoremstyle{bozont-remark-reverse}{3pt}{3pt}%
     {}
     {}
     {\sc}
     {.}
     {.5em}
     {\thmnumber{#2.}\thmname{ #1}\thmnote{ \rm #3}}
\newtheoremstyle{bozont-def-reverse}{3pt}{3pt}%
     {}
     {}
     {\bfseries}
     {.}
     {.5em}
     {\thmnumber{#2.}\thmname{ #1}\thmnote{ \rm #3}}
\newtheoremstyle{bozont-def-newnum-reverse}{3pt}{3pt}%
     {}
     {}
     {\bfseries}
     {}
     {.5em}
     {\thmnumber{#2.}\thmname{ #1}\thmnote{ \rm #3}}
\newtheoremstyle{bozont-def-newnum-reverse-plain}{3pt}{3pt}%
   {}
   {}
   {}
   {}
   {.5em}
   {\thmnumber{\!(#2)}\thmname{ #1}\thmnote{ \rm #3}}
\newtheoremstyle{bozont-number}{3pt}{3pt}%
   {}
   {}
   {}
   {}
   {0pt}
   {\thmnumber{\!(#2)} }
\newtheoremstyle{bozont-step}{3pt}{3pt}%
   {\itshape}
   {}
   {\scshape}
   {}
   {.5em}
   {$\boxed{\text{\sc \thmname{#1}~\thmnumber{#2}:\!}}$}
\theoremstyle{plain}    
  \basetheorem{proclaim}{Theorem}{}
\maketheorem{thm}{proclaim}{Theorem}
\maketheorem{mainthm}{proclaim}{Main Theorem}
\maketheorem{cor}{proclaim}{Corollary} 
\maketheorem{cors}{proclaim}{Corollaries} 
\maketheorem{lem}{proclaim}{Lemma} 
\maketheorem{prop}{proclaim}{Proposition} 
\maketheorem{conj}{proclaim}{Conjecture}
\basetheorem{subproclaim}{Theorem}{proclaim}
\maketheorem{sublemma}{subproclaim}{Lemma}

\maketheorem{subthm}{equation}{Theorem}
\maketheorem{subcor}{equation}{Corollary} 
\maketheorem{subprop}{equation}{Proposition} 
\maketheorem{subconj}{equation}{Conjecture}
\maketheorem{namedthm}{proclaim}{}
\newtheorem{proclaim-special}[proclaim]{\specialthmname}

\theoremstyle{bozont-subremark}
\maketheorem{subrem}{equation}{Remark}
\maketheorem{subnotation}{equation}{Notation} 
\maketheorem{subassume}{equation}{Assumptions} 
\maketheorem{subobs}{equation}{Observation} 
\maketheorem{subexample}{equation}{Example} 
\maketheorem{subex}{equation}{Exercise} 
\maketheorem{inclaim}{equation}{Claim} 
\maketheorem{subquestion}{equation}{Question}
\theoremstyle{bozont-remark}
\basetheorem{subremark}{Remark}{proclaim}
\makeremark{subclaim}{subremark}{Claim}
\maketheorem{rem}{proclaim}{Remark}
\maketheorem{claim}{proclaim}{Claim} 
\maketheorem{notation}{proclaim}{Notation} 
\maketheorem{assume}{proclaim}{Assumptions} 
\maketheorem{obs}{proclaim}{Observation} 
\maketheorem{example}{proclaim}{Example} 
\maketheorem{examples}{proclaim}{Examples} 
\maketheorem{complem}{equation}{Complement}
\maketheorem{const}{proclaim}{Construction}   
\maketheorem{ex}{proclaim}{Exercise}

\maketheorem{Fact}{proclaim}{Fact}

\newtheorem*{SubHeading*}{\SubHeadingName}%
\newtheorem{SubHeading}[proclaim]{\SubHeadingName}
\newtheorem{sSubHeading}[equation]{\sSubHeadingName}
\newenvironment{demo-r}[1]{\def\SubHeadingName{#1}\begin{SubHeading-r}}
  {\end{SubHeading-r}}%
\newenvironment{subdemo-r}[1]{\def\sSubHeadingName{#1}\begin{sSubHeading-r}}
  {\end{sSubHeading-r}} %
\newenvironment{demo*}[1]{\def\SubHeadingName{#1}\begin{SubHeading*}}
  {\end{SubHeading*}}%
\maketheorem{defini}{proclaim}{Definition}
\maketheorem{defnot}{proclaim}{Definitions and notation}
\maketheorem{question}{proclaim}{Question}
\maketheorem{terminology}{proclaim}{Terminology}
\maketheorem{crit}{proclaim}{Criterion}
\maketheorem{pitfall}{proclaim}{Pitfall}
\maketheorem{addition}{proclaim}{Addition}
\maketheorem{principle}{proclaim}{Principle} 
\maketheorem{condition}{proclaim}{Condition}
\maketheorem{exmp}{proclaim}{Example}
\maketheorem{hint}{proclaim}{Hint}
\maketheorem{exrc}{proclaim}{Exercise}
\maketheorem{prob}{proclaim}{Problem}
\maketheorem{ques}{proclaim}{Question}    
\maketheorem{alg}{proclaim}{Algorithm}
\maketheorem{remk}{proclaim}{Remark}          
\maketheorem{note}{proclaim}{Note}            
\maketheorem{summ}{proclaim}{Summary}         
\maketheorem{notationk}{proclaim}{Notation}   
\maketheorem{warning}{proclaim}{Warning}  
\maketheorem{defn-thm}{proclaim}{Definition--Theorem}  
\maketheorem{convention}{proclaim}{Convention}  
\maketheorem{hw}{proclaim}{Homework}
\maketheorem{hws}{proclaim}{\protect{${\mathbb\star}$}Homework}
\newtheorem*{ack}{Acknowledgment}

\maketheorem{defn}{proclaim}{Definition}
\maketheorem{subdefn}{equation}{Definition}
\maketheorem{corr}{proclaim}{Corollary} 
\maketheorem{lemr}{proclaim}{Lemma} 
\maketheorem{propr}{proclaim}{Proposition} 
\maketheorem{conjr}{proclaim}{Conjecture}
\newtheorem{SubHeading-r}[proclaim]{\SubHeadingName}
\newtheorem{sSubHeading-r}[equation]{\sSubHeadingName}

\newtheorem{proclaimr-special}[proclaim]{\specialthmname}
{\def\specialthmname{#1}\begin{proclaimr-special}}%
{\end{proclaimr-special}}
\maketheorem{remr}{proclaim}{Remark}
\maketheorem{subremr}{equation}{Remark}
\maketheorem{notationr}{proclaim}{Notation} 
\maketheorem{assumer}{proclaim}{Assumptions} 
\maketheorem{obsr}{proclaim}{Observation} 
\maketheorem{exampler}{proclaim}{Example} 
\maketheorem{exr}{proclaim}{Exercise} 
\maketheorem{claimr}{proclaim}{Claim} 
\maketheorem{inclaimr}{equation}{Claim} 
\maketheorem{definir}{proclaim}{Definition}
\maketheorem{newnumr}{proclaim}{}
\maketheorem{newnumrp}{proclaim}{}
\maketheorem{defnr}{proclaim}{Definition}
\maketheorem{questionr}{proclaim}{Question}
\newtheorem{newnumspecial}[proclaim]{\specialnewnumname}

\newtheorem{bstep}{Step}
\newcounter{thisthm} 
\newcounter{thissection} 

\newcommand{\iref}[1]{%
  (\the\value{thissection}.\the\value{thisthm}.\ref{#1})}
\newcounter{lect}
\newcounter{topic}
\counterwithin{equation}{proclaim}
\counterwithin{figure}{section}

\newenvironment{enumerate-p}{
  \begin{enumerate}}
  {\setcounter{equation}{\value{enumi}}\end{enumerate}}
\newenvironment{enumerate-cont}{
  \begin{enumerate}
    {\setcounter{enumi}{\value{equation}}}}
  {\setcounter{equation}{\value{enumi}}
  \end{enumerate}}

\newlength{\swidth}
\makeatother

\newcounter{stepp}
\newcommand{\mypagesize}{
\textwidth= 6.5in
\textheight=8.75in
\voffset-.5in
\hoffset-.75in
\marginparwidth=56pt
\footskip.5in
}

\mypagesize

\begin{document}

\title{Frobenius split anticanonical divisors}%
\author{S\'andor J Kov\'acs}%
\date{February 3, 2019} 
\thanks{Supported in part by NSF Grant DMS-1565352 and the Craig McKibben and Sarah
  Merner Endowed Professorship in Mathematics at the University of Washington.} %
\address{University of Washington, Department of Mathematics, Seattle, WA 98195, USA} 
\email{skovacs@uw.edu\xspace}
\urladdr{http://www.math.washington.edu/$\sim$kovacs\xspace}
\maketitle

\centerline{\it Dedicated to Emma Previato on the occasion of her 65th birthday}

\begin{abstract}
  In this note I extend two theorems of Sommese regarding abelian varieties to
  arbitrary characteristic; that an abelian variety cannot be an ample divisor in a
  smooth projective variety and that a cone over an abelian variety of dimension at
  least two is not smoothable.
\end{abstract}


\section{Introduction}
\noindent
The main goal of this note is to extend to arbitrary characteristic two theorems of
Sommese regarding abelian varieties; that an abelian variety cannot be an ample
divisor in a smooth projective variety \cite{MR0404703} and that a cone over an
abelian variety of dimension at least two is not smoothable \cite{MR522037}. Note
that the latter statement has already been extended to arbitrary characteristic in
\cite{KollarKovacs18b}, but the proof given here is different and arguably more
direct.  I give a new proof and a slightly stronger version of both results of
Sommese already in the characteristic zero case.
The main technical ingredient in positive characteristic is a sort of lifting theorem
(for the definition of \Fsp see \autoref{def:F-split}):
\begin{thm}[cf.~\autoref{cor:F-split-lifts-from-anti-canonical}]
  \label{thm:lift-intro} 
  Let $X$ be a smooth projective variety over an algebraically closed field $k$ of
  $\kar k>0$ and $D\subseteq X$ an effective anti-canonical divisor, i.e., 
  such that
  $\omega_X\simeq \sO_X(-D)$. If $D$ is \Fsp then so is $X$.
\end{thm}

\noin
This in turn is used to prove the vanishing of several cohomology groups:
\begin{thm}[cf.~\autoref{thm:F-split-anti-canonical-implies-ACM}]
  Let $X$ be a smooth projective variety over an algebraically closed field $k$ of
  $\kar k>0$ and $D\subseteq X$ an effective ample divisor such that
  $\omega_D\simeq \sO_D$. If $D$ is \Fsp and $\dim X\geq 3$, then
  \begin{enumerate}
  \item $H^i(X,\sO_X)=0$ for $i>0$, and
  \item $H^j(D,\sO_D)=0$ for $0<j<\dim D$.
  \end{enumerate}
\end{thm}
\begin{rem}
  See \autoref{cor:F-split-lifts-from-anti-canonical} and
  \autoref{thm:F-split-anti-canonical-implies-ACM} for stronger versions of these
  statements.
\end{rem}

\noin Finally, these vanishing results and Kawamata-Viehweg vanishing
\cite{Kawamata82c,Viehweg82b} in characteristic zero are used to prove characteristic
independent versions of Sommese's theorems in \autoref{cor:sommese} and
\autoref{cor:sommese-two}.

\begin{ack}
  I would like to thank Max Lieblich for useful conversations 
  and the referee for helpful comments.
  
  This work was supported in part by NSF Grant DMS-1565352 and the Craig McKibben and
  Sarah Merner Endowed Professorship in Mathematics at the University of Washington.
\end{ack}

\section{Frobenius splitting and vanishing}
\noin
The following notation will be used throughout the article. 

\begin{notation}\label{notation}
  Let $k$ be an algebraically closed field and $X$ a scheme over $k$. If
  $\kar k= p>0$ then let $F:X\to X$ denote the (absolute) Frobenius morphism. Recall
  that $F$ is the identity on the underlying space $X$ and its comorphism on the
  sheaf of regular functions is the $p^\text{th}$ power map: $\sO_X\to F_*\sO_X$
  given by $f\mapsto f^p$.
\end{notation}

\begin{defini}\label{def:F-split}
  A scheme $X$ over $k$ of $\kar k>0$ is called \emph{Frobenius split} or
  \emph{\Fsp} if the natural morphism $\eta:\sO_X\to F_*\sO_X$ has a left inverse,
  i.e., $\exists \eta':F_*\sO_X\to \sO_X$ such that $\eta'\circ\eta=\id_{\sO_X}$.
\end{defini}

It was proved in \cite[Prop.~2]{MR799251} that Kodaira vanishing \cite{Kodaira53}
holds on smooth projective \Fsp varieties. In fact, Mehta-Ramanathan's proof works in
a slightly more general setting:

\begin{thm}[Mehta-Ramanathan]\label{thm:kod-van-for-Fsp}
  Let $X$ be an equidimensional projective \CM scheme over $k$ of $\kar k>0$
  and let $\sL$ be an ample line bundle on $X$. If $X$ is \Fsp, then
  \[
  \xymatrix{%
    H^j(X, 
    \sL^{-1})=0 }
  \]
  for $j<\dim X$.
\end{thm}

\begin{proof}
  This follows directly from Serre duality \cite[7.6(b)]{Hartshorne77} and
  \cite[Prop.~1]{MR799251}.
\end{proof}


The proof of the following simple lemma uses the usual trick of obtaining a more
precise vanishing statement from Serre vanishing and surjective maps.

\begin{lem}\label{lem:vanishing-from-divisor-to-ambient}
  Let $X$ be an equidimensional projective \CM scheme over $k$ (of arbitrary
  characteristic) and $D\subseteq X$ an effective, ample Cartier divisor.  Fix an
  $m_0\in\bZ$ and a $j\in\bN$ such that $j<\dim X$ and let $\sE$ be a locally free
  sheaf on $X$.  Assume that $H^j(D,(\sE(-mD))\resto D)=0$ for each $m\geq m_0$.
  Then $H^j(X,\sE(-mD))=0$ for each $m\geq m_0$.
\end{lem}

\begin{proof}
  Consider the following \ses:
  \[
  \xymatrix{%
    0\ar[r] & \sE(-(m+1)D)\ar[r] & \sE(-mD) \ar[r] &  
    (\sE(-mD))\resto D \ar[r] &  0.
  }
  \]
  It follows from the assumption that the induced morphism
  \[
  \xymatrix{%
    H^j(X,\sE(-(m+1)D))\ar@{->>}[r] & H^j(X,\sE(-mD)) }
  \]
  is surjective for each $m\geq m_0$. By iterating this step we obtain that the
  induced morphism
  \[
  \xymatrix{%
    H^j(X,\sE(-(m+l)D))\ar@{->>}[r] & H^j(X,\sE(-mD)) }
  \]
  is surjective for any $l\in\bN$. However, $H^j(X,\sE(-(m+l)D))=0$ for
  $l\gg 0$ by Serre duality \cite[7.6(b)]{Hartshorne77} which implies the desired
  statement.
\end{proof}

\begin{cor}\label{cor:KV-for-anti-canonical}
  Let $X$ be an equidimensional projective Gorenstein scheme over $k$ and
  $D\subseteq X$ an effective, ample Cartier divisor such that
  $\omega_D\simeq \sO_D$.  Assume that if $\kar k>0$ then $D$ is \Fsp and if
  $\kar k=0$ then $X$ has rational \sings.  Then
  \[
  \xymatrix{%
    H^j(X,\omega_X(-mD))= 0 }
  \]
  for $j<\dim X-1$ and each $m\in\bN$.
\end{cor}


\begin{proof}
  First note that by the adjunction formula the assumption implies that
  $\omega_X(D)\resto D\simeq \sO_D$.  Fix a $j<\dim X-1=\dim D$.  Then
  \begin{multline*}
    H^{j}(D,\omega_X(-mD)\resto D)\simeq H^{j}(D,\left(\omega_X(D)\otimes
      \sO_X(-(m+1)D)\right)\resto D)\simeq \\ \simeq H^{j}(D,\sO_D(-(m+1)D\resto D))
    =0
  \end{multline*}
  for each $m\in \bN$ by Kawamata-Viehweg vanishing \cite{Kawamata82c,Viehweg82b} in
  case $\kar k=0$ and by \autoref{thm:kod-van-for-Fsp} if $\kar k>0$. Hence the
  statement follows from \autoref{lem:vanishing-from-divisor-to-ambient} by taking
  $\sE=\omega_X$ and $m_0=0$.
\end{proof}

\begin{thm}\label{thm:omega-is-minus-A}
  Let $X$ be an equidimensional projective Gorenstein scheme over $k$ and
  $D\subseteq X$ an effective, ample Cartier divisor such that
  $\omega_D\simeq \sO_D$.  Assume that $\dim X\geq 3$ and if $\kar k>0$ then assume
  further that $D$ is \Fsp. Then $\omega_X\simeq \sO_X(-D)$.
\end{thm}

\begin{proof}
  Consider the following \ses:
  \begin{equation*}
    \xymatrix{%
      0\ar[r] & \omega_X \ar[r] & \omega_X(D) \ar[r] & \omega_D\simeq \sO_D \ar[r] &
      0. 
    }
  \end{equation*}
  Observe that $H^j(X,\omega_X)=0$ for $j<\dim X-1$ by
  \autoref{cor:KV-for-anti-canonical}. In particular this holds for $j=1$, i.e.,
  $H^1(X,\omega_X)=0$, and so the induced morphism
  \[
  \xymatrix{%
    H^0(X,\omega_X(D)) \ar@{->>}[r] & H^0(D,\omega_D)\simeq H^0(D,\sO_D)
  }
  \]
  is surjective. It follows that there exists a section
  $0\neq s\in H^0(X,\omega_X(D))$ such that $(s=0)\cap D=\emptyset$.  Because $D$ is
  ample, this implies that $(s=0)=\emptyset$ and hence that
  $\omega_X(D)\simeq \sO_X$. This proves the statement.
\end{proof}

\section{Lifting Frobenius splittings}
\noin
The following is a simple criterion for Frobenius splitting, probably well-known to
experts. A proof is included for the convenience of the reader.

\begin{prop}\label{prop:crit-for-F-split}
  Let $X$ be a projective \CM scheme of equidimension $n$ over $k$ of $\kar
  k>0$.
  Then $X$ is \Fsp if and only if there exists a morphism
  $\sigma:\omega_X\to F_*\omega_X$ such that the induced morphism
  \[
  \xymatrix{%
    H^n(\sigma): H^n(X,\omega_X) \ar[r]^-{\neq 0} & H^n(X,F_*\omega_X) }
  \]
  is non-zero. Any morphism $\sigma$ satisfying the above criterion will be a called
  a \emph{dual splitting morphism of $X$}.
\end{prop}

\begin{rem}
  Note that an important feature of this criterion is that there is no assumption of
  functoriality or any other constraints on $\sigma$, only that $H^n(\sigma)\neq 0$.
\end{rem}

\begin{proof}
  If $X$ is \Fsp, then letting $\sigma$ be the Grothendieck dual of the splitting
  morphism $\eta':F_*\sO_X\to \sO_X$ given by the definition 
  shows the ``only if'' part of the claim.

  To show the other direction, first notice that since both $H^n(X,\omega_X)$ and
  $H^n(X,F_*\omega_X)$ are $1$-dimensional, the assumption is equivalent to saying
  that $H^n(\sigma)$ is an isomorphism.

  Next observe that by Serre duality the morphism
  \[
  H^n(\tau): H^n(X,F_*\omega_X)\to H^n(X,\omega_X)
  \]
  induced by the Grothendieck trace map $\tau: F_*\omega_X\to \omega_X$ is also
  non-zero, and again, since both $H^n(X,\omega_X)$ and $H^n(X,F_*\omega_X)$ are
  $1$-dimensional, it is an isomorphism.

  It follows that the composition $\tau\circ\sigma:\omega_X\to \omega_X$, which
  factors through $F_*\omega_X$, induces an isomorphism on $H^n(X,\omega_X)$. In
  particular, $\tau\circ\sigma\neq 0$.  Now let $\eta:\sO_X\to F_*\sO_X$ be the
  comorphism of the Frobenius, which is of course the Grothendieck dual of $\tau$ and
  let $\eta': F_*\sO_X\to \sO_X$ be the Grothendieck dual of $\sigma$. Then we see
  that $\eta'\circ\eta:\sO_X\to \sO_X$ cannot be zero, since otherwise so would be
  its Grothendieck dual, $\tau\circ\sigma$. However, if $\eta'\circ\eta\neq 0$, then
  it must be an isomorphism. Replacing $\eta'$ with itself composed with the inverse
  of this isomorphism we obtain that $X$ is \Fsp.
\end{proof}

\begin{defini}\label{def:dual-spl-morph-lifted}
  Let $X$ be an equidimensional projective \CM scheme over $k$ of $\kar k>0$ and
  $D\subseteq X$ a non-empty effective Cartier divisor. Assume that $D$ is \Fsp and
  let $\alpha:\omega_X(D)\to \omega_D$ denote the adjunction morphism. Then we have
  the following diagram:
  \begin{equation}
    \label{eq:1}
    \begin{aligned}
      \xymatrix@C=3em{%
        \omega_X(D) \ar@{..>}[d]_?^{\lambda} \ar[r]^-\alpha & \omega_D \ar[d]^\sigma \\
        F_*\left(\omega_X(D)\right) \ar[r]^-{F_*\alpha} & F_*\omega_D. \\
      }
    \end{aligned}
  \end{equation}
  Here $\sigma:\omega_D\to F_*\omega_D$ is a dual splitting morphism of $D$ provided
  by \autoref{prop:crit-for-F-split}. Note that the morphism $\lambda$ does not
  always exist. If a morphism $\lambda$ making the diagram commutative does exist
  then we will say that the \emph{dual splitting morphism $\sigma$ can be lifted to
    $X$}.
\end{defini}

\begin{thm}\label{thm:lifted-sigma-implies-F-split}
  Let $X$ be an equidimensional projective \CM scheme over $k$ of $\kar k>0$ and
  $D\subseteq X$ an effective (non-empty) Cartier divisor. Assume that $D$ is \Fsp
  and that $\sigma:\omega_D\to F_*\omega_D$, a dual splitting morphism of $D$, can be
  lifted to $X$ (cf.~\autoref{def:dual-spl-morph-lifted}). Then $X$ is also \Fsp.
\end{thm}

\begin{proof}
  I will use the notation of \autoref{def:dual-spl-morph-lifted}.  Let
  $\lambda:\omega_X(D)\to F_*\left(\omega_X(D)\right)$ be a lifting of $\sigma$,
  i.e., such that $\lambda$ completes \autoref{eq:1} to a commutative diagram.  Note
  that $\ker\alpha\simeq \omega_X$ and hence setting
  $\sigma_X\leteq \lambda|_{\omega_{\tiny X}}$ we have the following commutative
  diagram of \sess:
  \begin{equation*}
    \begin{aligned}
      \xymatrix{%
        0 \ar[r] & \omega_X \ar@{..>}[d]^{\sigma_X}
        \ar[r] & \omega_X(D) \ar[r] \ar[d]^\lambda & \omega_D \ar[r] \ar[d]^\sigma & 0 \\
        0 \ar[r] & F_*\omega_X \ar[r] & F_*\left(\omega_X(D)\right) \ar[r] &
        F_*\omega_D \ar[r] & 0. }
    \end{aligned}
  \end{equation*}
  Let $n\leteq \dim X$ and consider the induced diagram of long exact sequences of
  cohomology:
  \[
  \xymatrix@C4ex{%
    \dots \ar[r] & H^{n-1}(D,\omega_D) \ar[d]_{H^{n-1}(\sigma)} \ar[r] &
    H^{n}(X,\omega_X) \ar[d]^{H^n(\sigma_X)}\ar[r] & H^{n}(X,\omega_X(D))
    \\
    \dots \ar[r] & H^{n-1}(D,F_*\omega_D) \ar[r] & H^{n}(X,F_*\omega_X) \ar[r] &
    H^{n}(X,F_*\left(\omega_X(D)\right)) 
  }
  \]
  Observe that $H^{n}(X,F_*\left(\omega_X(D)\right))\simeq H^{n}(X,\omega_X(D))=0$ by
  Serre duality and hence we have a commutative square where the horizontal maps are
  isomorphisms:
  \[
  \xymatrix@C=3em{%
    H^{n-1}(D,\omega_D) \ar[d]_{H^{n-1}(\sigma)} \ar[r]^-\simeq & H^{n}(X,\omega_X)
    \ar[d]^{H^n(\sigma_X)} \\
    H^{n-1}(D,F_*\omega_D) \ar[r]^-\simeq & H^{n}(X,F_*\omega_X). }
  \]
  This implies that $H^n(\sigma_X)\neq 0$, hence $X$ is \Fsp by
  \autoref{prop:crit-for-F-split}.
\end{proof}

\begin{cor}\label{cor:F-split-lifts-from-anti-canonical}
  Let $X$ be an equidimensional projective Gorenstein scheme over $k$ of $\kar k>0$
  and $D\subseteq X$ an effective (non-empty) anti-canonical divisor, i.e., such that
  $\omega_X\simeq \sO_X(-D)$. If $D$ is \Fsp then so is $X$.
\end{cor}

\begin{proof}
  Since $D$ is an anti-canonical divisor it follows that $\omega_X(D)\simeq \sO_X$
  and $\omega_D\simeq \sO_D$, so the adjunction \ses
  \[
  \xymatrix{%
    0 \ar[r] & \omega_X \ar[r] & \omega_X(D) \ar[r] & \omega_D \ar[r] & 0 }
  \] 
  becomes
  \[
  \xymatrix{%
    0 \ar[r] & \sO_X(-D) \ar[r] & \sO_X \ar[r] & \sO_D \ar[r] & 0 }
  \]
  If $D$ is \Fsp, then the natural morphism $\eta_D:\sO_D\to F_*\sO_D$ splits and
  hence $H^n(\eta_D)\neq 0$. In other words, $\eta_D$ is a dual splitting morphism of
  $D$. The natural morphism $\eta_X:\sO_X\to F_*\sO_X$ is a lifting of $\eta_D$ to
  $X$ and hence $X$ is \Fsp by \autoref{thm:lifted-sigma-implies-F-split}.
\end{proof}

\begin{cor}\label{cor:KV-for-Fanos}
  Let $X$ be an equidimensional projective Gorenstein scheme over $k$ of $\kar k>0$
  and $D\subseteq X$ an effective (non-empty) anti-canonical divisor, i.e., such that
  $\omega_X\simeq \sO_X(-D)$. If $D$ is \Fsp then Kodaira vanishing holds on $X$,
  i.e., for any ample line bundle $\sL$ on $X$,
  \[
  \xymatrix{%
    H^j(X, \sL^{-1})=0 }
  \]
  for $j<\dim X$.
\end{cor}

\begin{proof}
  This is a direct consequence of \autoref{cor:F-split-lifts-from-anti-canonical} and
  \autoref{thm:kod-van-for-Fsp}. 
\end{proof}

\section{Frobenius split anti-canonical divisors on Fano varieties}
\noin
We are now ready to prove the main result.

\begin{thm}\label{thm:F-split-anti-canonical-implies-ACM}
  Let $X$ be an equidimensional projective Gorenstein scheme over $k$ and
  $D\subseteq X$ an effective, ample Cartier divisor such that
  $\omega_D\simeq \sO_D$.  Assume that $\dim X\geq 3$ and also that if $\kar k>0$
  then $D$ is \Fsp and if $\kar k=0$ then $X$ has rational \sings.  Then
  \begin{enumerate}
  \item\label{item:5} if $\kar k>0$ then $X$ is \Fsp,
  \item\label{item:7} Kodaira vanishing holds on $X$,
  \item\label{item:3} $\omega_X\simeq \sO_X(-D)$,
  \item\label{item:4} $H^i(X,\sO_X)=0$ for $i>0$, and
  \item\label{item:6} $H^j(D,\sO_D)=0$ for $0<j<\dim D$.
  \end{enumerate}
\end{thm}

\begin{proof}
  \autoref{thm:omega-is-minus-A} implies \autoref{item:3}, which combined with
  \autoref{cor:F-split-lifts-from-anti-canonical} implies \autoref{item:5} and if
  $\kar k>0$ then that combined with \autoref{thm:kod-van-for-Fsp} implies
  \autoref{item:7}.  Of course, if $\kar k=0$ then \autoref{item:7} is well-known by
  Kawamata-Viehweg vanishing \cite{Kawamata82c,Viehweg82b}.  Since
  $\sO_X\simeq \omega_X\otimes \sO_X(D)$ and $D$ is ample, Serre duality and
  \autoref{item:7} implies \autoref{item:4}. Finally,
  consider the \ses,
  \[
  \xymatrix{%
    0 \ar[r] & \sO_X(-D) \ar[r] & \sO_X \ar[r] & \sO_D \ar[r] & 0.
  }
  \]
  As $H^j(X,\sO_X(-D))=0$ for $j<\dim X$ by \autoref{item:7}, \autoref{item:6}
  follows from \autoref{item:4}.
\end{proof}

As a consequence of \autoref{thm:F-split-anti-canonical-implies-ACM} we will obtain
the generalization of the main result of \cite{MR0404703} to ordinary abelian
varieties in positive characteristic promised in the introduction.  Note that by
\cite[1.1]{MR916481} an abelian variety is ordinary if and only if it is \Fsp. In
particular, the methods of this paper do not say anything about what happens for
non-ordinary abelian varieties.

For the definition of ordinary varieties in general the reader is referred to
\cite{MR849653} although the definition of ordinariness will not be used directly.
We will only use the following properties, proved respectively by Illusie
\cite{MR1106904}, and Joshi and Rajan \cite{MR1936581}.

\begin{prop}\label{prop:facts-about-ordinary-varieties}
  Let $Z$ be an ordinary smooth projective variety over a field $k$ of positive
  characteristic. Then
  \begin{enumerate}
  \item\label{item:1} \cite[Prop.~1.2]{MR1106904} any small deformation of $Z$ is
    also ordinary, and
  \item\label{item:2} \cite[Prop.~3.1]{MR1936581} if in addition
    $\omega_Z\simeq \sO_Z$, then $Z$ is \Fsp.
  \end{enumerate}
\end{prop}

\begin{cor}\label{cor:sommese}
  Let $A$ be an abelian variety of dimension at least $2$ over $k$.  If $\kar k>0$
  assume that $A$ is ordinary. Suppose $A$ is an ample divisor on $X$. Then $X$
  cannot be an equidimensional projective Gorenstein scheme if $\kar k>0$ and an
  equidimensional projective Gorenstein scheme with only rational singularities if
  $\kar k=0$.
\end{cor}

\begin{proof}
  If $\kar k>0$ then $A$ is \Fsp by \cite[1.1]{MR916481} or
  \autoref{prop:facts-about-ordinary-varieties}\autoref{item:2}.  By the dimension
  assumption $H^1(A,\sO_A)\neq 0$ and hence the statement follows from
  \autoref{thm:F-split-anti-canonical-implies-ACM}\autoref{item:6}.
\end{proof}

\begin{rem}
  If $\kar k=0$, essentially the same statement as \autoref{cor:sommese} was proved
  in \cite{MR0404703}. However, the proof given here is quite different even in the
  $\kar k=0$ case. In particular, \cite{MR0404703} relied on topological arguments
  and it only stated that $A$ could not be an ample divisor on a smooth projective
  variety.
\end{rem}

\section{Non-smoothable singularities}
\noin
In this section I prove Sommese's second theorem which is an application of
\autoref{thm:F-split-anti-canonical-implies-ACM} showing that certain singularities
are not smoothable.


Next recall that the Betti numbers of smooth projective varieties are defined as the
dimension of $\ell$-adic cohomology groups, i.e., let $Z$ be a smooth projective
variety over 
$k$ and let $\ell$ be a prime different from $\kar k$. Then
$b_i(Z)\leteq \dim H^i_{\acute{e}t}(Z,\bQ_\ell)$.  Note that this definition is valid
in all characteristics. In characteristic zero these numbers are the same as the
dimension of singular cohomology groups of the underlying topological space of $Z$,
which is the more common definition in this case.

The following lemma is folklore. For the reader's convenience a proof is provided. It
also follows from various stronger statements which we will not need here. 

\begin{lem}\label{lem:h1-not-zero}
  Let $Z$ be a smooth projective variety over an algebraically closed field $k$. If
  $b_1(Z)\neq 0$ then $H^1(Z,\sO_Z)\neq 0$.
\end{lem}

\begin{proof}
%
  If $b_1(Z)\neq 0$ then by the expression of \'etale cohomology as a limit implies
  that there are arbitrarily high order torsion elements in $\Pic Z$
  (cf.~\cite[4.4.4]{MR1317816}). It follows that then $\Pic^\circ Z$ cannot be finite
  by the Theorem of the Base \cite{Thm_of_the_base,MR0052154}. Hence $\Pic^\circ Z$
  is positive dimensional and then so is $H^1(Z,\sO_Z) = T_{[\sO_Z]}\Pic^\circ Z$.
\end{proof}

I will use the notion of smoothability used in \cite{MR522037} and
\cite{MR0349677}. This is slightly more restrictive than the one used in
\cite{KollarKovacs18b}.

\begin{defini}[cf.~\cite{MR0349677,MR522037})]\label{def:smoothability}
  For a morphism $f:\sX\to T$ and $t\in T$, the fibre of $f$ over $t$ will be denoted
  by $\sX_t= f^{-1}(t)$.
  Let $X$ be a closed subscheme of a scheme $P$ over an algebraically closed field
  $k$. A \emph{deformation of $X$ in $P$} is a morphism $(f:\sX\to T)$ where
  \begin{enumerate}
  \item $T$ is a connected positive dimensional scheme of finite type over $k$,
  \item $\sX\subseteq P\times T$ is a closed subscheme which is flat over $T$,
  \item there exists a closed point $0\in T$, such that $\sX_0\simeq X$, and
  \item $f$ is the restriction of the projection morphism $P\times T\to T$ to $\sX$.
  \end{enumerate}
  A deformation $(f:\sX\to T)$ of $X$ in $P$ will be called a \emph{Gorenstein
    deformation} if for all $t\in T$, $t\neq 0$, the fibre $\sX_t$ is Gorenstein.  It
  will be called a \emph{smooth deformation} if for all $t\in T$, $t\neq 0$, the
  fibre $\sX_t$ is smooth over $k(t)$. In this latter case we also say that $X$ is
  \emph{smoothable in $P$}.
\end{defini}

A somewhat weaker statement than the following was proved in \cite[2.1.1]{MR522037}
in characteristic zero and a somewhat more general statement, as a consequence of
much deeper results, was established in \cite[Cor.~8.7]{KollarKovacs18b} in all
characteristics.

\begin{thm}\label{thm:non-smoothable}
  Let $X\subseteq \bP^n$ be a projective variety over an algebraically closed field
  $k$ and let $H\subseteq \bP^n$ be a hypersurface such that $Z=X\cap H$ is smooth,
  $\dim Z>1$, $\omega_Z\simeq \sO_Z$, and $b_1(Z)\neq 0$ (e.g., $Z$ is an abelian
  variety of dimension at least $2$). Further assume that if $\kar k>0$ then $Z$ is
  ordinary. Let $(f:\sX\to T)$ be a deformation of $X$ in $\bP^n$ and assume that if
  $\kar k=0$ then $\sX_t$ has rational \sings for $t\neq 0$. Then there exists a
  non-empty open neighbourhood $0\in U\subseteq T$ such that $\sX_t$ has isolated
  non-Gorenstein singularities for every $t\in U$.  Consequently $X$ does not admit a
  Gorenstein deformation in $\bP^n$ and in particular, it is not smoothable in
  $\bP^n$.
\end{thm}

\begin{proof}
  Observe that there exists a non-empty open set 
  $0\in U\subseteq T$ such that $\sX_t$ has the same properties as $X=\sX_0$, that
  is, there exists a hypersurface $H_t\subseteq \bP^n$ such that $Z_t=\sX_t\cap H_t$
  is smooth, $\dim Z_t>1$, $\omega_{Z_t}\simeq \sO_{Z_t}$, and $b_1(Z_t)\neq 0$. We
  may also assume that if $\kar k>0$ then $Z_t$ is ordinary by
  \autoref{prop:facts-about-ordinary-varieties}\autoref{item:1} and note that if
  $\kar k=0$ then $\sX_t$ has rational singularities for $t\neq 0$ by assumption.

  It follows that $Z_t$ is \Fsp by
  \autoref{prop:facts-about-ordinary-varieties}\autoref{item:2} and
  $H^1(Z_t,\sO_{Z_t})\neq 0$ by \autoref{lem:h1-not-zero} and hence $\sX_t$ is not
  Gorenstein by \autoref{thm:F-split-anti-canonical-implies-ACM}\autoref{item:6}. 

  On the other hand, as a hypersurface section of $\sX_t$ is smooth, its singular set
  must be zero-dimensional and hence $\sX_t$ must have isolated non-Gorenstein
  singularities.
\end{proof}

\noin Finally, this implies the following:

\begin{cor}\label{cor:sommese-two}
  Let $A\subseteq \bP^{n-1}$ be an abelian variety of dimension at least $2$ over $k$
  and let $X$ denote the cone over $A$ in $\bP^{n}$.  If $\kar k>0$ further assume
  that $A$ is ordinary. Then $X$ is not smoothable in $\bP^{n}$.
\end{cor}

\end{document}